\documentclass[10pt]{amsart}

\usepackage[T1]{fontenc}
\usepackage[english]{babel}
\usepackage{csquotes}

\usepackage[a4paper,margin=3cm]{geometry}
\allowdisplaybreaks
\usepackage{graphicx}
\usepackage[dvipsnames]{xcolor}
\usepackage{caption}
\usepackage{subcaption}
\usepackage{tikz}
\usetikzlibrary{
  babel,
  positioning,
  arrows.meta
}
\usepackage{tikz-cd}
\usepackage{pgfplots}
\pgfplotsset{compat=1.18}

\usepackage{mathtools}
\usepackage{amssymb}
\usepackage{esint}
\usepackage{mathrsfs}
\usepackage{faktor}
\usepackage{dsfont}

\usepackage[shortlabels]{enumitem}
\usepackage{array}
\usepackage{hhline}
\usepackage[normalem]{ulem}
\usepackage{comment}
\usepackage[toc,page]{appendix}
\usepackage{imakeidx}
\usepackage{fancyhdr}
\usepackage{ifthen}
\usepackage{forloop}
\usepackage{xstring}
\usepackage{emptypage}
\usepackage{listings}

\usepackage[
  backend=biber,
  style=numeric,
  sorting=nty,
  giveninits=true,
  maxnames=10,
  doi=false,
  isbn=false,
  url=false,
  language=auto,
  autolang=other
]{biblatex}

\usepackage[
  hyperfootnotes=false
]{hyperref}

\hypersetup{
  colorlinks=true,
  linkcolor=blue,
  citecolor=blue,
  urlcolor=red
}

\usepackage{aliascnt}
\usepackage[
  nameinlink,
  capitalise,
  sort
]{cleveref}

\crefname{equation}{}{}
\Crefname{equation}{}{}
\creflabelformat{equation}{#2(#1)#3}
\crefname{enumi}{}{}
\Crefname{enumi}{}{}

\newcommand*{\newsharedthm}[2]{%
  \newaliascnt{#1}{lemma}%
  \newtheorem{#1}[#1]{#2}%
  \aliascntresetthe{#1}%
}

\theoremstyle{plain}
\newtheorem{lemma}{Lemma}[section]
\newsharedthm{theorem}{Theorem}
\newsharedthm{proposition}{Proposition}
\newsharedthm{corollary}{Corollary}

\theoremstyle{definition}
\newsharedthm{definition}{Definition}
\newsharedthm{conjecture}{Conjecture}
\newsharedthm{claim}{Claim}
\newsharedthm{assumption}{Assumption}

\theoremstyle{remark}
\newsharedthm{remark}{Remark}
\newsharedthm{example}{Example}
\newsharedthm{notation}{Notation}

\newcommand*{\setcrefname}[3]{%
  \crefname{#1}{#2}{#3}%
  \Crefname{#1}{#2}{#3}%
}

\setcrefname{lemma}{Lemma}{Lemmas}
\setcrefname{theorem}{Theorem}{Theorems}
\setcrefname{proposition}{Proposition}{Propositions}
\setcrefname{corollary}{Corollary}{Corollaries}
\setcrefname{definition}{Definition}{Definitions}
\setcrefname{conjecture}{Conjecture}{Conjectures}
\setcrefname{claim}{Claim}{Claims}
\setcrefname{assumption}{Assumption}{Assumptions}
\setcrefname{remark}{Remark}{Remarks}
\setcrefname{example}{Example}{Examples}
\setcrefname{notation}{Notation}{Notations}

\newlist{thmenum}{enumerate}{1}
\setlist[thmenum]{
  label=(\roman*),
  ref=\thetheorem(\roman*)
}
\setcrefname{thmenumi}{Theorem}{Theorems}

\numberwithin{equation}{section}

\newcommand{\R}{\mathbb{R}}
\newcommand{\C}{\operatorname{C}}
\renewcommand{\H}{\operatorname{H}}
\newcommand{\B}{\operatorname{B}}
\newcommand{\supp}{\operatorname{supp}}
\newcommand{\PP}{\mathbb{P}}
\DeclareRobustCommand{\Bdot}{\dot{\B}\protect{\vphantom{B}}}
\DeclareRobustCommand{\Cdot}{\dot{\C}\protect{\vphantom{C}}}
\DeclareRobustCommand{\Hdot}{\dot{\H}\protect{\vphantom{H}}}

\begin{document}

\title[Liouville theorem for 2D stationary fractional Navier--Stokes]{A Liouville theorem for the two-dimensional stationary hypodissipative Navier--Stokes system}



\author[N.~De Nitti]{Nicola De Nitti}
\address[N.~De Nitti]{Politecnico di Bari, Dipartimento di Meccanica, Matematica e Management, Via E.~Orabona 4, 70125 Bari, Italy.}
\email[]{nicola.denitti@poliba.it}

\author[L.~Niebel]{Lukas Niebel}
\address[L.~Niebel]{Institut f\"ur Analysis und Numerik, Universit\"at M\"unster\\
  Orl\'eans-Ring 10, 48149 M\"unster, Germany.}
\email[]{lukas.niebel@uni-muenster.de}

\author[J.~Yang]{Jiaqi Yang}
\address[J.~Yang]{Northwestern Polytechnical University, School of Mathematics and Statistics, Youyi West Road 127, Beilin District
  Xi’an, 710072 Shaanxi, P.\,R.~China.}
\email[]{yjqmath@nwpu.edu.cn}

\begin{abstract}
We study the two-dimensional stationary incompressible Navier--Stokes equations on \(\R^2\) with fractional dissipation \((-\Delta)^s\). In the full range \(s \in (0,1)\), we prove that every smooth solution satisfying the natural energy condition \(u\in\Hdot^s(\R^2;\R^2)\) has \(u\equiv0\) and constant pressure. This is a fractional counterpart
of the planar finite-Dirichlet theorem of Gilbarg and Weinberger at \(s=1\).
The proof uses different arguments in three ranges. For \(0<s<\frac13\), we combine an \(\mathrm L^2\)-estimate derived from the equation with a stream-function truncation argument. For \(\frac13\leq s\leq\frac23\), we use a localized energy estimate whose boundary terms are supported on expanding annuli. For \(\frac23<s<1\), we establish regularity and decay via a Lorentz-space bootstrap and then apply the maximum principle to the vorticity. We also treat the stationary damped Euler system at \(s=0\) by combining the Bernoulli identity with a cut-off argument under an annular growth condition that includes \(u\in L^r(\R^2)\) for every \(1\le r\le2\).
\end{abstract}

\maketitle

\section{Introduction}
\label{sec:introduction}

The fractional, or hypodissipative, Navier--Stokes system replaces the classical
Laplacian by a fractional Laplace operator of order \(2s\in(0,2)\). It therefore
interpolates between the damped Euler system, where no viscous diffusion is present, and
the classical Navier--Stokes system, corresponding to \(s=1\). In this paper we
study the two-dimensional incompressible stationary hypodissipative Navier--Stokes
system on the whole plane:
\begin{equation}\label{eq:NS}
  \begin{cases}
    u\cdot\nabla u + (-\Delta)^s u + \nabla p = 0, & x\in\R^2, \\
    \operatorname{div} u =0,                       & x\in\R^2,
  \end{cases}
\end{equation}
where \(0<s<1\). Here, $u = u(x) \colon \R^2 \to \R^2$ is the velocity field and $p \colon \R^2 \to \R$ is the pressure. For sufficiently regular functions,
\[
  (-\Delta)^s u(x)
  = c_{2,s}\,\mathrm{p.v.}\!\int_{\R^2}
  \frac{u(x)-u(y)}{|x-y|^{2+2s}}\,\mathrm d y,
\]
where \(c_{2,s}\) is chosen so that the Fourier representation
\[
  \widehat{(-\Delta)^s u}
  = |\xi|^{2s}\,\widehat u
\]
holds in \(\R^2\); see \cite{BuVa2016,MR3967804}.

The natural homogeneous energy associated with \cref{eq:NS} is
\[
  [u]_{\Hdot^s(\R^2)}^2
  \coloneqq
  \|(-\Delta)^{s/2}u\|_{\mathrm L^2(\R^2)}^2
  =
  \frac{c_{2,s}}{2}
  \iint_{\R^2\times\R^2}
  \frac{|u(x)-u(y)|^2}{|x-y|^{2+2s}}\,\mathrm d x\,\mathrm d y .
\]
Throughout the paper, \(\Hdot^s(\R^2)\) denotes the standard homogeneous Sobolev
space realized inside \(\mathrm L^{q_s}(\R^2)\), where
\(
q_s \coloneqq \frac{2}{1-s}.
\)
In particular, non-zero constants do not belong
to \(\Hdot^s(\R^2)\). By the homogeneous Sobolev embedding in dimension two
(see \cite[Theorem~6.1.3]{MR2463316} for the strong
\(\mathrm L^{q_s}\)-embedding and \cite[Theorem~4.10 and eq.~(1.8)]{MR223874} for
the Lorentz refinement),
\begin{equation}\label{eq:SobolevEmbedding}
  \Hdot^s(\R^2)
  \hookrightarrow
  \mathrm L^{q_s,2}(\R^2)
  \hookrightarrow
  \mathrm L^{q_s}(\R^2),
  \qquad
  q_s = \frac{2}{1-s}.
\end{equation}

The main subject of this paper is the study of Liouville-type theorems for
\cref{eq:NS}. These are rigidity results asking whether a solution defined on the
whole space must be trivial under natural decay, integrability, or finite-energy
assumptions.

For the hypodissipative system \cref{eq:NS}, this question is particularly
delicate because the fractional diffusion is weaker. The main issue is to determine whether the natural energy
condition \(u\in \Hdot^s(\R^2)\), together with the stationary equation and the
incompressibility constraint, is strong enough to rule out nontrivial solutions.

The main result of this note is the following Liouville-type theorem that covers the full range $s \in (0,1)$.

\begin{theorem}\label{thm:main-bounded}
  Let $s \in (0,1)$ and let \((u,p)\in\C^\infty(\R^2;\R^2)\times\C^\infty(\R^2)\) be a smooth solution of \cref{eq:NS}. 
  If
  \[
    u\in \Hdot^s(\R^2;\R^2),
  \]
  then $u\equiv 0$ and \(p\) is constant.
\end{theorem}

The corresponding Liouville theorem also holds at the zero-order endpoint $s=0$:
\[
  (-\Delta)^0=I,
  \qquad
  \Hdot^0(\R^2)=\mathrm L^2(\R^2),
\]
and \cref{eq:NS} becomes the stationary damped Euler system
\begin{equation}\label{eq:euler}
\begin{cases}
  u\cdot\nabla u+u+\nabla p=0, & x \in \mathbb R^2, 
  \\
  \operatorname{div} u=0, & x \in \mathbb R^2.
  \end{cases}
\end{equation}

\begin{theorem}\label{thm:damped-euler}
  Let \((u,p)\in\C^\infty(\R^2;\R^2)\times\C^\infty(\R^2)\) be a smooth solution of \cref{eq:euler}. Assume that
  \begin{equation}\label{eq:damped-annular-condition}
    \liminf_{R\to\infty}\frac1R
    \int_{B_{2R}\setminus B_R}|u(x)|\,\mathrm dx=0.
  \end{equation}
  Then \(u\equiv0\) and \(p\) is constant. In particular, the conclusion holds if
  \(u\in\mathrm L^r(\R^2;\R^2)\) for some \(1\le r\le2\).
\end{theorem}

\subsection{Survey of the literature}
\label{ssec:lit}

Liouville-type theorems for stationary Navier--Stokes systems have their classical
starting point in the planar case \(s=1\). Gilbarg and Weinberger proved that a
steady two-dimensional solution with finite Dirichlet integral must have constant
velocity \cite{MR501907}. See also the later two-dimensional works of Kozono,
Terasawa, and Wakasugi, which further develop the vorticity and decay mechanisms
in related Liouville problems \cite{MR3797615,MR4448585}.

The three-dimensional case is more delicate; many classical Liouville criteria impose additional integrability, decay, or structural hypotheses; see \cite{MR3289443} and the references therein. It is remarked in \cite[Remark X.9.4]{MR2808162} and \cite{MR3538409}
that it remains an open problem to prove that a solution of the three-dimensional stationary Navier--Stokes equations satisfying
\[
  u \in \Hdot^{1}(\mathbb{R}^3;\R^3)
  \qquad\text{and}\qquad
  u(x) \to 0 \quad \text{as } |x| \to +\infty
\]
must be identically zero. Nevertheless, under additional assumptions, results of
the following form are known: if
\(
u \in \Hdot^{1}(\mathbb{R}^3) \cap \mathrm X(\mathbb{R}^3)
\),
then \(u\equiv 0\).

Several choices of the auxiliary space \(X\) have been considered. In \cite[Theorem~X.9.5]{MR2808162}, the case
\(X=\mathrm L^{9/2}(\mathbb{R}^3)\) was treated;
the result of \cite{MR3538409} assumes
\(u\in\mathrm L^6\cap\mathrm{BMO}^{-1}\), and hence applies when
\(u\in\Hdot^1\cap\mathrm{BMO}^{-1}\);
and in \cite{MR4227049},
\(X=\mathrm L^q(\mathbb{R}^3)\) for some
\(
3 \le q \le \frac{9}{2},
\)
without assuming \(u\in \Hdot^1(\mathbb{R}^3)\). By the classical Sobolev
embedding,
\(
\Hdot^{1}(\mathbb{R}^3) \hookrightarrow \mathrm L^{6}(\mathbb{R}^3),
\)
but this information alone appears to be insufficient to conclude that a solution
\(u\in \Hdot^{1}(\mathbb{R}^3)\) is trivial.

In the fractional case \(s\in(0,1)\), earlier results were concentrated mainly in three dimensions and often
imposed additional global Lebesgue assumptions. In \cite{MR3808580}, Wang and Xiao treated an \(n\)-dimensional stationary fractional compressible Navier--Stokes--Poisson system using the Caffarelli--Silvestre extension (see \cite{MR2354493}). In the
constant-density planar case, their theorem gives triviality under $ u\in\Hdot^s(\R^2)\cap\mathrm L^2(\R^2)$ for $\frac12\leq s<1$, 
and under $u\in\Hdot^s(\R^2)\cap\mathrm L^2(\R^2)\cap\mathrm L^6(\R^2)$ for $0<s<\frac12$
Their incompressible three-dimensional corollary
assumes
\(u\in\Hdot^s(\R^3)\cap\mathrm L^{9/2}(\R^3)\) for \(0<s<1\).  In \cite{MR4450182}, Yang proved that, for \(5/6\le s<1\), a smooth \(L^{9/2}(\R^3)\)-solution that tends uniformly to
zero is trivial; no \(\Hdot^s\)-assumption is required

In \cite{MR4839126}, Chamorro and Poggi prove \(u\equiv0\) under the assumption
\(u\in\Hdot^s(\R^3)\cap
\mathrm L^{(6-\varepsilon)/(3-2s)}(\R^3)\) for for 
\(\frac12\leq s<1\); or \(u\in\Hdot^s(\R^3)\cap
\mathrm L^{(6-\varepsilon)/(3-2s)}(\R^3)\cap
\mathrm L^{(6+\varepsilon)/(3-2s)}(\R^3)\) for
\(\frac3{10}<s<\frac12\). In \cite{MR4813662}, Jarr\'in and Vergara-Hermosilla
proved \(u\equiv0\) when
\(u\in\Hdot^s(\R^3)\cap\mathrm L^p(\R^3)\), with \(\frac56<s\leq1\) and
\(\max\{3/(2s-1),3/s\}<p\leq9/2\).  As a Navier--Stokes
specialization of his fractional MHD results, Zeng
\cite{Zeng2025} proved \(u\equiv0\) when
\(\frac12\leq s\leq\frac56\) and
\(u\in\Hdot^s(\R^3)\). Tan
\cite{tan2025newliouvilletypetheorems} obtained,
in particular, \(u\equiv0\) when 
\(u\in\Hdot^s(\R^3)\cap
\Bdot^{1-2s}_{\infty,\infty}(\R^3)\), with \(\frac12<s<1\); at \(s=\frac56\), the result requires only
\(u\in\Hdot^{5/6}(\R^3)\).

In \cite{wang2025}, Wang, Yang, and Yu
proved that, for
\(\frac12<s<1\), the assumptions
\((-\Delta)^{s/2}u\in\mathrm L^2(\R^3)\),
\(u\in\dot{\mathrm W}^{1+2s,\infty}(\R^3)\), and
\(u(x)\to u_\infty\neq0\) uniformly as \(|x|\to\infty\) imply
\(u\equiv u_\infty\). In \cite{LT26}, Liu and Tan
removed the higher-derivative assumption:
if \((-\Delta)^{s/2}u\in\mathrm L^2(\R^3)\) and
\(u(x)\to u_\infty\) uniformly, then \(u\equiv u_\infty\) for
\(\frac12\leq s<1\) when \(u_\infty\neq0\), while
\(u\equiv0\) for \(\frac12\leq s\leq\frac56\) when
\(u_\infty=0\); when $u_\infty \equiv 0$, the argument extends to
\(\R^n\), \(n\geq2\), in the range
\(\frac12\leq s\leq(n+2)/6\)

In \cite{LeeLee2026}, Lee and Lee
proved  $u \equiv 0$, for
\(3\leq n\leq6\), assuming  either 
\(u\in\Hdot^s(\R^n)\), \(u(x)\to0\) uniformly as
\(|x|\to\infty\), for
\(n/6\leq s<(n+2)/6\) or 
\(\limsup_{R\to\infty}R^{-\lambda}
(\int_{B_{2R}\setminus B_R}|(-\Delta)^{s/2}u|^2\,\mathrm d x)^{1/2}
<\infty\) for some \(0\leq\lambda<s\), together with suitable
integrability assumptions on \(u\) or \(\nabla u\). Although \(n=2\)
is not included, the formal planar specialization of their
finite-energy range is \(\frac13\leq s<\frac23\).

For the stationary Euler equation with linear damping, the literature is more limited. In \cite{MR3627131}, Chepyzhov, Ilyin, and Zelik studied the time-dependent damped--driven Euler system on \(\R^2\); in particular, their energy identity applied to a time-independent solution with zero forcing, gives \(u\equiv0\) under the stronger assumption
\(u\in\mathrm H^1(\R^2)\). In \cite{MR3162482}, Chae
 proved that
$u\in\mathrm L^q(\R^2)$ and $
  p\in\mathrm L^{q/2}(\R^2)$, with 
$  q>6$, 
imply \(u\equiv0\). 

\subsection{Difficulties and outline of the proof}

Formally, testing \cref{eq:NS} against \(u\) would give
\(\|(-\Delta)^{s/2}u\|_{\mathrm L^2}^2=0\). The difficulty is to
justify this calculation under the sole assumption
\(u\in\Hdot^s(\R^2)\). Although the Sobolev embedding gives
\(u\in\mathrm L^{q_s}(\R^2)\), where \(q_s=2/(1-s)\), the homogeneous
energy assumption by itself does not provide \(\mathrm L^2\)-control
or sufficient pointwise decay. Thus \(u\) is not initially available
as a global test function.

A direct localization of the energy identity under the sole Sobolev
information works only in an intermediate range.   Control of the cubic boundary terms requires
\(q_s\geq3\), equivalently \(s\geq\frac13\), while the resulting
annular estimate contains the factor \(R^{3s-2}\) and therefore closes
only for \(s\leq\frac23\). The pressure must also be identified with
its Riesz-transform representative, up to an additive constant, before
the localized pressure term can be estimated.

These restrictions lead to three complementary arguments; the endpoint
\(s=0\) is treated separately.

\begin{description}
  \item[The range \(0<s<\frac13\)]
  The projected equation first yields \(u\in L^2(\R^2)\), and a second use
  of that equation controls the associated stream function at zero frequency:
  \(u=\nabla^\perp\psi\) with
  \(\psi\in H^{1+s}(\R^2)\cap C_0(\R^2)\).
  Truncating the values of \(\psi\) produces compactly supported,
  divergence-free test fields for which both pressure and convection vanish.
  The truncations converge weakly to \(u\) in \(H^s\), recovering the full
  energy identity. This method works throughout \(0<s<\frac12\).

  \item[The range \(\frac13\le s\le\frac23\)]
  After normalizing
  \(p=\mathcal R_i\mathcal R_j(u_i u_j)\), we test against \(u\chi_R\).
  The dissipative term converges to the full fractional energy, while the
  Bernoulli boundary flux is bounded by
  \(R^{3s-2}\) times vanishing annular norms. More generally, the same
  argument works for every \(0<s<1\) if
  \(u\in L^r(\R^2)\) for some \(3\le r\le6\).

  \item[The range \(\frac23<s<1\)]
  A Lorentz-space iteration applied to the projected equation gives
  boundedness. H\"older potential estimates then give
  \(u\in C_b^{1,\gamma}\) and the decay of \(u\) and \(\nabla u\).
  The vorticity is continuous and vanishes at infinity, so the nonlocal
  maximum principle yields \(\omega=0\), from which we deduce \(u\equiv0\).

  \item[The endpoint \(s=0\)]
  For the Bernoulli function \(B=p+\frac12|u|^2\), the damped Euler equation gives
  $
    \operatorname{div}\bigl(u\arctan B\bigr)
    =-\frac{|u|^2}{1+B^2}.
  $
  Testing against expanding cutoffs and using a suitable annular growth condition makes the boundary flux vanish and
  forces \(u=0\).
\end{description}

Some preliminary results are collected in \cref{sec:prelim}. The
stream-function, annular cut-off, and vorticity arguments are presented
in \cref{sec:stream-range,sec:proof-sub,sec:proof-super}, respectively,
and combined in \cref{sec:proofs}. The damped Euler endpoint is treated separately in \cref{sec:euler}.

Throughout the paper, $C$ denotes a positive constant whose value may change from line to line. Moreover, the symbols $\lesssim$, $\gtrsim$, and $\approx$ denote comparisons up to positive constants depending only on the fixed parameters under consideration. 

\section{Preliminaries}
\label{sec:prelim}
Taking divergence in \cref{eq:NS}, using $\operatorname{div}u=0$, we obtain
\[
  -\Delta p=\partial_i\partial_j(u_i u_j).
\]
For the estimates below we use the normalized pressure associated with the
velocity field,
\begin{equation}\label{eq:pressure}
  p_0\coloneqq
  (-\Delta)^{-1}\partial_i\partial_j(u_i u_j)
  =\mathcal R_i\mathcal R_j(u_i u_j),
\end{equation}
where \(\mathcal R_i\) denotes the \(i\)-th Riesz transform; see \cite[Proposition~5.1.17, Example~5.1.18, and Corollary~5.2.8]{MR3243734}.
Let us explain why we may work with this normalized pressure.
\begin{lemma}\label{lem:pressure-normalization}
  Let \(0<s<1\), let \(u\in \Hdot^s(\R^2;\R^2)\) be divergence-free, and let
  \((u,p)\) be a smooth solution of \cref{eq:NS} in \(\R^2\). Set
  \[
    q_s=\frac{2}{1-s},
    \qquad
    r_s=\frac{q_s}{2}=\frac{1}{1-s},
  \]
  and consider the normalized pressure defined in \cref{eq:pressure}.
  Then \(p-p_0\) is constant. In particular,
  \begin{equation}\label{eq:normalized-equation}
    (-\Delta)^s u+u\cdot\nabla u+\nabla p_0=0
    \qquad\text{in }\mathcal S'(\R^2;\R^2).
  \end{equation}
\end{lemma}

\begin{proof}
  By the homogeneous Sobolev embedding \cref{eq:SobolevEmbedding},
  \(u\in \mathrm L^{q_s}(\R^2)\). Hence
  \[
    u_i u_j\in \mathrm L^{r_s}(\R^2),
    \qquad
    r_s=\frac{1}{1-s}\in(1,\infty).
  \]
  Therefore Calder\'on--Zygmund theory gives
  \[
    p_0\in \mathrm L^{r_s}(\R^2),
    \qquad
    \|p_0\|_{\mathrm L^{r_s}}
    \lesssim
    \|u\otimes u\|_{\mathrm L^{r_s}}.
  \]

  Taking divergence in \cref{eq:NS} and using \(\operatorname{div}u=0\), we obtain
  \[
    -\Delta p=\partial_i\partial_j(u_i u_j).
  \]
  By the definition of \(p_0\), also
  \[
    -\Delta p_0=\partial_i\partial_j(u_i u_j).
  \]
  Thus \(h\coloneqq p-p_0\) is harmonic in the sense of distributions:
  \[
    \Delta h=0.
  \]

  We claim that \(\nabla h=0\). From \cref{eq:NS} and the identity
  \(u\cdot\nabla u=\operatorname{div}(u\otimes u)\), we have
  \[
    \nabla h
    =
    -(-\Delta)^s u
    -\operatorname{div}(u\otimes u)
    -\nabla p_0
    \qquad\text{in }\mathcal D'(\R^2;\R^2).
  \]
  Hence, for every \(\varphi\in C_c^\infty(\R^2;\R^2)\),
  \[
    \begin{aligned}
      |\langle \nabla h,\varphi\rangle|
       & \le
      [u]_{\Hdot^s}\,[\varphi]_{\Hdot^s}
      +
      \left(
      \|u\otimes u\|_{\mathrm L^{r_s}}
      +
      \|p_0\|_{\mathrm L^{r_s}}
      \right)
      \|\nabla\varphi\|_{\mathrm L^{r_s'}},
    \end{aligned}
  \]
  where \(r_s'=r_s/(r_s-1)=1/s\).

  The right-hand side shows in particular that \(\nabla h\) is a tempered
  distribution. Since \(\Delta h=0\), each component of \(\nabla h\) is a tempered
  harmonic distribution. Therefore \(\nabla h\) is a vector-valued harmonic
  polynomial.

  It remains to rule out non-zero polynomials. Suppose, for contradiction, that
  \(P\coloneqq\nabla h\not\equiv0\), and let \(d\ge0\) be its degree. Choose
  \(\varphi\in \C_c^\infty(\R^2;\R^2)\) such that, if \(P_d\) denotes the top
  homogeneous part of \(P\), then
  \[
    \int_{\R^2} P_d(y)\cdot \varphi(y)\,dy\ne0.
  \]
  For \(R>1\), set \(\varphi_R(x)=\varphi(x/R)\). Then
  \[
    |\langle P,\varphi_R\rangle|
    \ge cR^{d+2}
  \]
  for all sufficiently large \(R\). On the other hand, the previous estimate gives
  \[
    |\langle P,\varphi_R\rangle|
    \lesssim
    [\varphi_R]_{\Hdot^s}
    +
    \|\nabla\varphi_R\|_{\mathrm L^{r_s'}}
    \lesssim
    R^{1-s}+R^{2/r_s'-1}.
  \]
  Since \(r_s'=1/s\), this is
  \[
    |\langle P,\varphi_R\rangle|
    \lesssim
    R^{1-s}+R^{2s-1}.
  \]
  Both exponents \(1-s\) and \(2s-1\) are strictly smaller than \(1\), while
  \(d+2\ge2\). This is impossible as \(R\to\infty\). Hence \(P=0\), so
  \(\nabla h=0\), and therefore \(h\) is constant.

  Thus \(p-p_0\) is constant, and replacing \(p\) by \(p_0\) gives
  \cref{eq:normalized-equation}.
\end{proof}

From now on, whenever the pressure is estimated, we use this normalized pressure
and write $p$ instead of $p_0$.
In particular, whenever $u\otimes u\in \mathrm L^r(\R^2)$ with $1<r<\infty$,
\begin{equation}\label{eq:CZ-pressure}
  \|p\|_{\mathrm{L}^r(\R^2)} \lesssim \|u\otimes u\|_{\mathrm{L}^r(\R^2)}.
\end{equation}

We also use a standard truncation fact for homogeneous Sobolev spaces; see \cite[Lemma~B.1]{MR4225499}.

\begin{lemma}\label{lem:trunc}
  Let $0<s<1$, let \(m\in\mathbb N\), and
  {\(u\in\Hdot^s(\R^2;\R^m)\)}.  Let $\eta_R\in \C^\infty(\R^2)$ satisfy
  \[
    0\le \eta_R\le1,
    \qquad \eta_R\equiv0\text{ on }B_R,
    \qquad \eta_R\equiv1\text{ on }\R^2\setminus B_{2R},
    \qquad |\nabla\eta_R|\le C R^{-1},
  \]
  where \(C\) is independent of \(R\). Then
  \[
    [\eta_R u]_{\Hdot^s(\R^2)}\to0
    \qquad \text{as }R\to\infty.
  \]
\end{lemma}

The following elementary facts will be used in the high-\(s\) range, above
the threshold at which the annular cutoff estimate closes.

We write \(\mathrm{UC}(\R^2;\R^m)\) for the space of globally uniformly
continuous functions from \(\R^2\) to \(\R^m\).
\begin{lemma}\label{lem:decay-uc}
 Let \(f:\R^2\to\R^m\) for some \(m\in\mathbb N\).
  \begin{enumerate}
    \item[(a)] If \(1\le q<\infty\) and $f\in \mathrm{L}^q(\R^2;\R^m)\cap \mathrm{UC}(\R^2;\R^m)$, then $f(x)\to0$ as $|x|\to\infty$.
    \item[(b)] If $f\in \C^1(\R^2;\R^m)$, $f(x)\to0$ as $|x|\to\infty$, and $\nabla f$ is uniformly continuous on $\R^2$, then $\nabla f(x)\to0$ as $|x|\to\infty$.
  \end{enumerate}
\end{lemma}

\begin{proof}
  For part (a), suppose, for contradiction, that there exist $\varepsilon>0$ and $x_k\to\infty$ such that $|f(x_k)|\ge\varepsilon$. By uniform continuity, there exists $r>0$ such that $|f(x)-f(x_k)|\le\varepsilon/2$ whenever $x\in B_r(x_k)$; hence $|f(x)|\ge\varepsilon/2$ on every such ball. After passing to a subsequence, we may assume that these balls are pairwise disjoint. Then
  \[
    \int_{\R^2}|f|^q\, \mathrm{d} x \ge \sum_k\int_{B_r(x_k)}|f|^q\, \mathrm{d} x
    \ge \sum_k |B_r|\,(\varepsilon/2)^q = \infty,
  \]
  a contradiction.

  For part (b), we argue by contradiction again. Then there exist $\varepsilon>0$ and $x_k\to\infty$ such that $|\nabla f(x_k)|\ge\varepsilon$. Passing to a subsequence, we may assume that, for some fixed indices $i\in\{1,\dots,m\}$ and $j\in\{1,2\}$ and some fixed sign $\sigma\in\{-1,1\}$,
  \[
    \sigma\,\partial_j f_i(x_k)\ge \frac{\varepsilon}{\sqrt{2m}}
    \qquad \text{for all }k.
  \]
  By uniform continuity of $\nabla f$, there exists $r>0$ such that
  \[
    \sigma\,\partial_j f_i(x)\ge \frac{\varepsilon}{2\sqrt{2m}}
    \qquad \text{for every }x\in B_r(x_k),
  \]
  for all $k$. Hence, for each $k$,
  \[
    \sigma\bigl(f_i(x_k+\tfrac r2e_j)-f_i(x_k)\bigr)
    = \int_0^{r/2} \sigma\,\partial_j f_i(x_k+te_j)\, \mathrm{d} t
    \ge \frac{\varepsilon r}{4\sqrt{2m}}.
  \]
  Since both $x_k$ and $x_k+\frac r2e_j$ tend to infinity, this contradicts
  the assumed decay of $f$.
\end{proof}

We shall also use Lorentz spaces $\mathrm{L}^{p,r}(\R^2)$. The only product estimate needed below is the Lorentz-space H\"older inequality \cite[Theorem~4.5]{MR223874}:
\begin{equation}\label{eq:lorentz-holder}
  \|fg\|_{\mathrm{L}^{p/2,1}(\R^2)}
  \lesssim \|f\|_{\mathrm{L}^{p,2}(\R^2)}\,\|g\|_{\mathrm{L}^{p,2}(\R^2)},
  \qquad 2<p<\infty.
\end{equation}

For a matrix field \(F=(F_{\ell j})_{\ell,j=1}^2\), we use
\((\operatorname{div} F)_\ell=\partial_jF_{\ell j}\), and we denote by
\(\PP\) the \emph{Leray projector}.  If
\(F\in\mathrm L^r(\R^2;\R^{2\times2})\), with \(1<r<\infty\), we define
\[
  (\PP\operatorname{div} F)_i
  \coloneqq
  \partial_jF_{ij}+\partial_i\mathcal R_\ell\mathcal R_jF_{\ell j}.
\]
Equivalently, away from \(\xi=0\),
\[
  \widehat{(\PP\operatorname{div} F)}_i(\xi)
  =
  i\left(\delta_{i\ell}-\frac{\xi_i\xi_\ell}{|\xi|^2}\right)
  \xi_j\widehat F_{\ell j}(\xi).
\]
For every \(0<s<1\), the normalized equation and \cref{eq:pressure} give
\begin{equation} \label{eq:projected-equation}
  (-\Delta)^s u+\PP\operatorname{div}(u\otimes u)=0
  \qquad\text{in }\mathcal S'(\R^2;\R^2).
\end{equation}

In the range \(s>1/2\), set \(\alpha=2s-1\). The operator
\begin{equation}\label{eq:Tdef}
  T=(-\Delta)^{-s}\PP\operatorname{div}
\end{equation}
has Fourier multiplier
\[
  M_{i\ell j}(\xi)
  =
  i|\xi|^{-2s}
  \left(\delta_{i\ell}-\frac{\xi_i\xi_\ell}{|\xi|^2}\right)\xi_j,
  \qquad \xi\ne0.
\]
Thus \(M\) is smooth away from the origin and homogeneous of degree
\(-\alpha\). We regard \(T\) as an operator with values in
\(\mathcal S'(\R^2;\R^2)/\mathcal P\), where \(\mathcal P\) denotes the space of
vector-valued polynomials.

Consequently,
\begin{equation}\label{eq:u-equals-T}
  u=-T(u\otimes u)
  \qquad\text{in }\mathcal S'(\R^2;\R^2)/\mathcal P .
\end{equation}
Indeed, by the definition of \(T\), the Fourier transform of
\(u+T(u\otimes u)\) vanishes on \(\R^2\setminus\{0\}\). Hence it is supported at
\(\{0\}\), and the structure theorem for distributions with point support implies
that \(u+T(u\otimes u)\) is a vector-valued polynomial.

For \(0<\theta<1\), we write
\[
  [f]_{\Cdot^{0,\theta}}
  \coloneqq
  \sup_{x\ne y}\frac{|f(x)-f(y)|}{|x-y|^\theta},
  \qquad
  \mathrm C_b^{0,\theta}
  \coloneqq
  \mathrm L^\infty\cap\Cdot^{0,\theta}.
\]
We equip this space with the norm
\[
  \|f\|_{\mathrm C_b^{0,\theta}}
  \coloneqq
  \|f\|_{\mathrm L^\infty}+[f]_{\Cdot^{0,\theta}}.
\]
For \(0<\gamma<1\), we also set
\[
  \mathrm C_b^{1,\gamma}
  \coloneqq
  \left\{
  f\in \mathrm L^\infty:
  \nabla f\in \mathrm L^\infty,
  \ [\nabla f]_{\Cdot^{0,\gamma}}<\infty
  \right\}.
\]
Its norm is
\[
  \|f\|_{\mathrm C_b^{1,\gamma}}
  \coloneqq
  \|f\|_{\mathrm L^\infty}
  +\|\nabla f\|_{\mathrm L^\infty}
  +[\nabla f]_{\Cdot^{0,\gamma}}.
\]
The identity \cref{eq:projected-equation} is exact in \(\mathcal S'\). The
polynomial ambiguity appears only after applying the homogeneous inverse
\((-\Delta)^{-s}\). Every subsequent use of \cref{eq:u-equals-T} is therefore
understood as follows: choose a representative \(G\) of \(T(u\otimes u)\), then
\(u+G\) is a polynomial, and this polynomial is removed by
\cref{lem:polynomial-remainder}.

\begin{lemma}\label{lem:polynomial-remainder}
  Let $N\in\mathbb N$, $1\le p<\infty$, $0<\theta<1$, and suppose
  \[
    v+G=P
    \qquad\text{in }\mathcal S'(\R^2;\R^N),
  \]
  where $v\in \mathrm L^p(\R^2;\R^N)$, $P$ is an $\R^N$-valued polynomial, and
  $G\in \Cdot^{0,\theta}(\R^2;\R^N)$ modulo constants. Then $P$ is
  constant. Consequently $v\in\Cdot^{0,\theta}(\R^2;\R^N)$.

  The conclusion that $P$ is constant also holds if
  $G\in \mathrm L^\infty(\R^2;\R^N)$ instead of
  $G\in\Cdot^{0,\theta}(\R^2;\R^N)$. If instead
  $G\in\mathrm L^r(\R^2;\R^N)$ for some $1\le r<\infty$, then $P=0$.
\end{lemma}

\begin{proof}
  We prove the H\"older case first. Choose a representative of $G$ with $G(0)=0$.
  Then
  \[
    |G(x)|\le [G]_{\Cdot^{0,\theta}}|x|^\theta.
  \]
  If $P$ had degree $d\ge1$, then one of its components would have degree $d$.
  Hence, on some open cone $\Gamma\subset \R^2$ and for all sufficiently large
  $|x|$, one would have
  \[
    |P(x)|\ge c|x|^d.
  \]
  Since $d\ge1>\theta$, it would follow on $\Gamma$ and for all sufficiently
  large $|x|$ that
  \[
    |v(x)|=|P(x)-G(x)|\ge c'|x|^d,
  \]
  contradicting $v\in \mathrm L^p(\R^2;\R^N)$. Hence $P$ has degree zero.

  If $G\in\mathrm L^\infty$, the same cone argument excludes every non-constant
  polynomial, since the bounded term $G$ cannot cancel the growth of $P$ at
  infinity.

  Finally, suppose $G\in\mathrm L^r(\R^2;\R^N)$ with $1\le r<\infty$.
  We show that no non-zero polynomial $P$ is possible. Since
  $v+G=P$ in distributions and all three terms are locally integrable, the
  identity holds a.e.~after choosing representatives.

  Assume, for contradiction, that $P\not\equiv0$. Let $d\ge0$ be the degree of
  $P$. Then there exist an open cone $\Gamma\subset\R^2$, constants $c>0$ and
  $R_0>0$, such that
  \[
    |P(x)|\ge c|x|^d
    \qquad\text{for all }x\in\Gamma,\ |x|\ge R_0.
  \]
  For $R\ge R_0$, set
  \[
    A_R\coloneqq \Gamma\cap\{R<|x|<2R\}.
  \]
  Then $|A_R|\ge c_\Gamma R^2$ for some $c_\Gamma>0$, and on $A_R$ we have
  \[
    |v(x)|+|G(x)|\ge |P(x)|\ge cR^d
    \qquad\text{for a.e.~}x\in A_R.
  \]
  Decompose
  \[
    A_R^v\coloneqq
    \left\{x\in A_R:\ |v(x)|\ge \frac{c}{2}R^d\right\},
    \qquad
    A_R^G\coloneqq A_R\setminus A_R^v .
  \]
  On $A_R^G$ we have $|G(x)|\ge \frac{c}{2}R^d$. Hence
  \begin{align*}
    \int_{A_R}|v|^p\,\, \mathrm{d} x+\int_{A_R}|G|^r\,\, \mathrm{d} x
     & \ge
    \left(\frac c2\right)^p R^{dp}|A_R^v|
    +
    \left(\frac c2\right)^r R^{dr}|A_R^G| \\
     & \ge
    C R^{d\min\{p,r\}} |A_R|              \\
     & \ge
    C R^{2+d\min\{p,r\}} .
  \end{align*}
  This is impossible because each integral on the left is bounded by the
  corresponding tail integral over $\{|x|>R\}$, which tends to zero as
  $R\to\infty$. Therefore $P\equiv0$.
\end{proof}

The operator $T$ has order $-(2s-1)$.
We set
\[
  \alpha\coloneqq 2s-1\in(0,1).
\]
We shall use the following mapping properties.

\begin{lemma}\label{lem:mapping-T}
  Let $F:\R^2\to\R^{2\times2}$.
  \begin{enumerate}[label=(\textup{\alph*})]
    \item If $F\in\mathrm L^{m,r}(\R^2;\R^{2\times2})$,
          $1<m<2/\alpha$, and $1\le r\le\infty$, then
          \[
            \|T F\|_{\mathrm{L}^{n,r}(\R^2)} \lesssim \|F\|_{\mathrm{L}^{m,r}(\R^2)},
            \qquad \frac1n = \frac1m - \frac\alpha2.
          \]
    \item If $F\in \mathrm{L}^{2/\alpha,1}(\R^2)$, then
          \[
            \|T F\|_{\mathrm{L}^\infty(\R^2)} \lesssim \|F\|_{\mathrm{L}^{2/\alpha,1}(\R^2)}.
          \]
    \item If $F\in\mathrm L^m(\R^2;\R^{2\times2})$ with
          $m>2/\alpha$, then $TF$ is defined modulo constants and
          \[
            [TF]_{\Cdot^{0,\alpha-2/m}(\R^2)}
            \lesssim
            \|F\|_{\mathrm{L}^m(\R^2)}.
          \]
    \item Let $0<\beta<1$ and $\beta+\alpha\ne1$. Suppose that
          $v\in\mathrm L^\infty(\R^2;\R^2)$ satisfies
          \[
            v=-TF
            \qquad\text{in }\mathcal S'(\R^2;\R^2)/\mathcal P.
          \]
          If $F\in \mathrm C_b^{0,\beta}(\R^2;\R^{2\times2})$, then
          \[
            \|v\|_{\mathrm C_b^{0,\beta+\alpha}}
            \lesssim
            \|v\|_{\mathrm L^\infty}
            +
            \|F\|_{\mathrm C_b^{0,\beta}}
            \qquad\text{when } \beta+\alpha<1,
          \]
          while
          \[
            \|v\|_{\mathrm C_b^{1,\beta+\alpha-1}}
            \lesssim
            \|v\|_{\mathrm L^\infty}
            +
            \|F\|_{\mathrm C_b^{0,\beta}}
            \qquad\text{when } \beta+\alpha>1.
          \]
  \end{enumerate}
\end{lemma}

\begin{proof}
  Parts (a)--(b) are the Hardy--Littlewood--Sobolev and endpoint Lorentz estimates
  for Riesz potentials, combined with the boundedness of Calder\'on--Zygmund
  operators; see \cite[Theorem~4.10]{MR223874} and \cite[Theorem~2.6]{MR146673}.
  Part (c) is the corresponding Morrey--Sobolev estimate for potentials of order $\alpha$, again composed with
  Calder\'on--Zygmund operators; see \cite[Theorem~2.2]{MR1428685}.

  For background on Besov spaces, Littlewood--Paley decompositions, and the Besov--Hölder and multiplier estimates used below, we refer the reader to \cite[Sections~2.2 and~2.4.2]{MR2463316}. Let
  \[
    \sigma\coloneqq \beta+\alpha.
  \]
  The symbol of $T=(-\Delta)^{-s}\PP\operatorname{div}$ is smooth away from the origin and
  homogeneous of degree $-\alpha$. Let $(\Delta_j)_{j\ge -1}$ be a
  nonhomogeneous dyadic decomposition. Since
  \[
    v+TF\in\mathcal P
  \]
  and $\Delta_j$ kills polynomials for all $j\ge0$, we have
  \[
    \Delta_j v=-\Delta_j TF,
    \qquad j\ge0.
  \]
  The Fourier multiplier theorem in Besov spaces gives, for $j\ge0$,
  \[
    \|\Delta_j TF\|_{\mathrm L^\infty}
    \lesssim
    2^{-j\alpha}\|\widetilde\Delta_j F\|_{\mathrm L^\infty},
  \]
  where $\widetilde\Delta_j$ denotes a harmless enlargement of $\Delta_j$.
  Hence,
  \[
    2^{j\sigma}\|\Delta_j v\|_{\mathrm L^\infty}
    \lesssim
    2^{j\beta}\|\widetilde\Delta_j F\|_{\mathrm L^\infty},
    \qquad j\ge0.
  \]
  The low frequency is controlled by the assumed boundedness of $v$:
  \[
    \|\Delta_{-1}v\|_{\mathrm L^\infty}
    \lesssim
    \|v\|_{\mathrm L^\infty}.
  \]
  Therefore
  \[
    \|v\|_{\mathrm B_{\infty,\infty}^{\sigma}}
    \lesssim
    \|v\|_{\mathrm L^\infty}
    +
    \|F\|_{\mathrm B_{\infty,\infty}^{\beta}}.
  \]
  Since $0<\beta<1$, we have
  \[
    \mathrm B_{\infty,\infty}^{\beta}(\R^2)=\mathrm C_b^{0,\beta}(\R^2).
  \]
  If $\sigma<1$, then
  \[
    \mathrm B_{\infty,\infty}^{\sigma}(\R^2)
    =
    \mathrm C_b^{0,\sigma}(\R^2),
  \]
  while if $\sigma>1$ and $\sigma\notin\mathbb Z$, then
  \[
    \mathrm B_{\infty,\infty}^{\sigma}(\R^2)
    =
    \mathrm C_b^{1,\sigma-1}(\R^2).
  \]
  This proves the asserted estimates.
\end{proof}

We also use the boundedness of the modulus map below
\(\mathrm H^{3/2}\), established in \cite{MR3949429}, to obtain the following lemma.

\begin{lemma}[Truncation below $\mathrm H^{3/2}$]\label{lem:clamp}
  Let $0<\sigma<1/2$, $f\in \mathrm H^{1+\sigma}(\R^2)$, and
  $a>0$. Define the truncation
  \[
    \mathcal C_a(f)(x)\coloneqq
    \begin{cases}
      -a,   & f(x)<-a,     \\
      f(x), & |f(x)|\le a, \\
      a,    & f(x)>a.
    \end{cases}
  \]
  Then $\mathcal C_a(f)\in \mathrm H^{1+\sigma}(\R^2)$ and
  \begin{equation}\label{eq:clamp-bound}
    \|\mathcal C_a(f)\|_{\mathrm H^{1+\sigma}}
    \leq C_\sigma\|f\|_{\mathrm H^{1+\sigma}},
  \end{equation}
  where $C_\sigma$ is independent of $a$.
\end{lemma}

\begin{proof}
  Choose \(0\leq\eta\in\mathrm C_c^\infty(\R^2)\) such that
  \(\eta=1\) on \(B_1\), and set \(\eta_R(x)=\eta(x/R)\). By scaling,
  \[
    [\eta_R]_{\Hdot^{1+\sigma}}
    =R^{-\sigma}[\eta]_{\Hdot^{1+\sigma}}.
  \]
  For \(a>0\), define
  \[
    h_R^\pm\coloneqq |f\pm a\eta_R|-a\eta_R.
  \]
  We use the absolute-value estimate
  \[
    [|g|]_{\Hdot^{1+\sigma}}
    \leq C_\sigma[g]_{\Hdot^{1+\sigma}},
    \qquad
    g\in\mathrm H^{1+\sigma}(\R^2),
  \]
  from \cite[Introduction, Item~3]{MR3949429}. Applying it to
  \(g=f\pm a\eta_R\) and using the triangle inequality, we obtain
  \[
    [h_R^\pm]_{\Hdot^{1+\sigma}}
    \leq
    C_\sigma[f]_{\Hdot^{1+\sigma}}
    +C_{\sigma,\eta}aR^{-\sigma}.
  \]
  The reverse triangle inequality also gives
  \[
    |h_R^\pm|\leq |f|.
  \]
  Since \(\eta_R\to1\) pointwise, Lebesgue's dominated convergence theorem yields
  \[
    h_R^\pm\longrightarrow |f\pm a|-a
    \qquad\text{in }\mathrm L^2(\R^2).
  \]
  The preceding estimates make \(h_R^\pm\) bounded in
  \(\mathrm H^{1+\sigma}(\R^2)\). After passing to a weakly convergent
  subsequence, the strong \(\mathrm L^2\)-convergence identifies its
  weak limit as \(|f\pm a|-a\). Hence weak lower semicontinuity gives
  \[
    [|f\pm a|-a]_{\Hdot^{1+\sigma}}
    \leq C_\sigma[f]_{\Hdot^{1+\sigma}},
  \]
  where \(C_\sigma\) is independent of \(a\).

  Finally, the pointwise identities
  \[
    \mathcal C_a(f)
    =
    \frac12\bigl((|f+a|-a)-(|f-a|-a)\bigr),
    \qquad
    |\mathcal C_a(f)|\leq |f|,
  \]
  imply
  \[
    [\mathcal C_a(f)]_{\Hdot^{1+\sigma}}
    \leq C_\sigma[f]_{\Hdot^{1+\sigma}},
    \qquad
    \|\mathcal C_a(f)\|_{\mathrm L^2}
    \leq \|f\|_{\mathrm L^2}.
  \]
  This proves \cref{eq:clamp-bound}.
\end{proof}

\section{\texorpdfstring{Stream-function truncation for $0<s<\frac12$}{Stream-function truncation for 0<s<1/2}}
\label{sec:stream-range}

The argument applies for every \(0<s<\frac12\), although the main theorem
uses it only below \(s=\frac13\); the overlapping range is also covered by
\cref{sec:proof-sub}. The projected equation first yields
\(u\in\mathrm L^2(\R^2)\); then it is used a second time to control the stream function at zero frequency and obtain a decaying
\(\psi\in\mathrm H^{1+s}(\R^2)\). Truncating the values of \(\psi\) produces
compactly supported, divergence-free test fields for which both the pressure
and convection terms vanish. Removing the truncation recovers the full
\(\Hdot^s\)-energy and forces \(u\equiv0\).

\begin{lemma}[$\mathrm L^2$ gain]\label{lem:L2gain}
  Let $0<s<1/2$.  If $u\in\Hdot^s(\R^2;\R^2)$ is divergence free and
  satisfies \cref{eq:projected-equation},
  then $u\in\mathrm L^2(\R^2;\R^2)$.
\end{lemma}

\begin{proof}
  \uline{Step 1:} \emph{Fourier control furnished by the equation.}
  Set \(q_s=2/(1-s)\). By \cref{eq:SobolevEmbedding} and
  Hausdorff--Young,
  \[
    u\otimes u\in\mathrm L^{q_s/2},
    \qquad
    \widehat{u\otimes u}\in\mathrm L^{1/s}.
  \]
  Taking the Fourier transform of \cref{eq:projected-equation}, we obtain for
  almost every $\xi\in\R^2\setminus\{0\}$
  \begin{equation}\label{eq:u-fourier-low}
    \widehat u_i(\xi)
    =
    -i|\xi|^{-2s}
    \left(\delta_{ik}-\frac{\xi_i\xi_k}{|\xi|^2}\right)
    \xi_j\widehat{u_k u_j}(\xi),
    \qquad
    |\widehat u(\xi)|
    \leq
    C|\xi|^{1-2s}
    |\widehat{u\otimes u}(\xi)|.
  \end{equation}

  \uline{Step 2:} \emph{Square integrability.} Since \(s<1/2\), the multiplier on the right is bounded on \(B_1\),
  and \(1/s>2\). Hence
  \[
    \widehat u\in\mathrm L^{1/s}(B_1)\subset\mathrm L^2(B_1).
  \]
  At high frequencies, the homogeneous energy gives
  \[
    \int_{\R^2\setminus B_1}|\widehat u|^2\,\mathrm d\xi
    \le \int_{\R^2\setminus B_1}|\xi|^{2s}|\widehat u|^2\,\mathrm d\xi
    <\infty.
  \]
  Thus \(\widehat u\in\mathrm L^2\), and Plancherel yields
  \(u\in\mathrm L^2(\R^2;\R^2)\).
\end{proof}

\begin{lemma}[Stream function]\label{lem:streamfunction}
  Let $0<s<1/2$, and let
  $u\in\Hdot^s(\R^2;\R^2)$ be divergence free and satisfy
  \cref{eq:projected-equation}. Then there exists a real-valued function
  $\psi\in\mathrm H^{1+s}(\R^2)$ such that
  \begin{equation}\label{eq:stream}
    u=\nabla^\perp\psi
    \qquad\text{a.e. in }\R^2.
  \end{equation}
  Moreover, $\psi\in\mathrm C^0(\R^2)$ and
  \begin{equation}\label{eq:psi-decay}
    \psi(x)\to 0
    \qquad\text{as }|x|\to\infty.
  \end{equation}
\end{lemma}

\begin{proof}
  \uline{Step 1:} \emph{Construction of the stream function.}
  By \cref{lem:L2gain}, $u\in\mathrm L^2(\R^2;\R^2)$.  For $\xi\neq0$,
  define
  \begin{equation}\label{eq:psi-def}
    \widehat\psi(\xi)
    :=
    -i\frac{\xi^\perp\cdot\widehat u(\xi)}{|\xi|^2}.
  \end{equation}
  Set $\widehat\psi(0)=0$.

  \uline{Step 2:} \emph{Square integrability at high and low frequencies.} Since $\operatorname{div} u=0$,
  \[
    \xi\cdot\widehat u(\xi)=0
    \ \text{for a.e. }\xi \quad
    \text{ and } \quad
    i\xi^\perp\widehat\psi(\xi)=\widehat u(\xi)
    \ \text{for a.e. }\xi\neq0.
  \]
  For $|\xi|\geq1$,
  \[
    |\widehat\psi(\xi)|
    \leq
    |\xi|^{-1}|\widehat u(\xi)|
    \leq
    |\widehat u(\xi)|,
  \]
  and hence the high-frequency part of $\widehat\psi$ belongs to
  $\mathrm L^2$.

  Since $u\in\mathrm L^2$, we have $u\otimes u\in\mathrm L^1$ and
  \[
    \|\widehat{u\otimes u}\|_{\mathrm L^\infty}
    \leq
    C\|u\otimes u\|_{\mathrm L^1}
    \leq
    C\|u\|_{\mathrm L^2}^2.
  \]
Taking the scalar product of the Fourier transform of
  \cref{eq:projected-equation} with
  \(-i\xi^\perp/|\xi|^2\) gives, for almost every \(\xi\neq0\), 
  \begin{equation}\label{eq:psi-equation-fourier}
    |\xi|^{2s}\widehat\psi(\xi)
    =
    -\frac{\xi_i^\perp\xi_j}{|\xi|^2}
    \widehat{u_i u_j}(\xi).
  \end{equation}
  We then deduce
  \[
    |\widehat\psi(\xi)|
    \leq
    C|\xi|^{-2s}\|u\|_{\mathrm L^2}^2,
    \qquad |\xi|\leq1.
  \]
    Since
  \[
    \int_{|\xi|\leq1}|\xi|^{-4s}\,\mathrm d\xi<\infty
    \quad\Longleftrightarrow\quad s<\frac12,
  \]
  \(\widehat\psi\in\mathrm L^2(\R^2)\), and hence
  $\psi\in\mathrm L^2(\R^2)$.  Since $u$ is real valued,
  $\widehat u(-\xi)=\overline{\widehat u(\xi)}$, and \cref{eq:psi-def} gives
  $\widehat\psi(-\xi)=\overline{\widehat\psi(\xi)}$. Hence $\psi$ is real
  valued. The Fourier identity above therefore proves \cref{eq:stream}
  a.e.~in $\R^2$.

  \uline{Step 3:} \emph{Regularity, continuity, and decay.}
  By \cref{eq:stream},
  \[
    \|(-\Delta)^{(1+s)/2}\psi\|_{\mathrm L^2}
    =
    \|(-\Delta)^{s/2}u\|_{\mathrm L^2}
    <\infty.
  \]
  Together with $\psi\in\mathrm L^2$, this yields
  \[
    \psi\in\mathrm H^{1+s}(\R^2).
  \]
  Furthermore, Cauchy--Schwarz gives
  \[
    \begin{aligned}
      \|\widehat\psi\|_{\mathrm L^1}
       & \leq
      \left(
      \int_{\R^2}(1+|\xi|^2)^{-1-s}\,\, \mathrm{d}\xi
      \right)^{1/2}
      \left(
      \int_{\R^2}
      (1+|\xi|^2)^{1+s}|\widehat\psi(\xi)|^2\,\, \mathrm{d}\xi
      \right)^{1/2} \\
       & \leq
      C_s\|\psi\|_{\mathrm H^{1+s}}
      <\infty.
    \end{aligned}
  \]
  Fourier inversion and the Riemann--Lebesgue lemma now imply that
  $\psi\in\mathrm C^0(\R^2)$ and that \cref{eq:psi-decay} holds.
\end{proof}

The field
\(z_\varepsilon=\nabla^\perp(\psi-\mathcal C_\varepsilon(\psi))\) is
compactly supported and divergence free, and agrees almost everywhere with
either \(u\) or \(0\). Testing the equation against this field eliminates both
the pressure and convection terms. Letting \(\varepsilon\downarrow0\) then
gives \(\|(-\Delta)^{s/2}u\|_{\mathrm L^2}=0\), which proves the result.

\begin{proposition}\label{prop:stream-range}
  Let $0<s<1/2$, let $(u,p)$ be a smooth solution of \cref{eq:NS} on
  $\R^2$, and suppose that
  \[
    u\in\Hdot^s(\R^2;\R^2).
  \]
  Then $u\equiv0$.
\end{proposition}

\begin{proof}
  \uline{Step 1:} \emph{Truncation along the stream-function levels.}
  \cref{lem:L2gain,lem:streamfunction} give
  \[
    u\in\mathrm L^2(\R^2;\R^2),
    \qquad
    \psi\in\mathrm H^{1+s}(\R^2),
    \qquad
    u=\nabla^\perp\psi,
    \qquad
    \psi(x)\to 0
    \quad\text{as }|x|\to\infty.
  \]
  In particular,
  \[
    u\in\mathrm H^s(\R^2;\R^2).
  \]

  For $\varepsilon>0$, set
  \begin{equation}\label{eq:vz}
    v_\varepsilon
    :=
    \nabla^\perp \mathcal C_\varepsilon(\psi),
    \qquad
    z_\varepsilon
    :=
    u-v_\varepsilon
    =
    \nabla^\perp\bigl(\psi-\mathcal C_\varepsilon(\psi)\bigr).
  \end{equation}
  Both \(v_\varepsilon\) and \(z_\varepsilon\) are divergence free. For every $c\in\R$, a Sobolev function satisfies
  \[
    \nabla\psi=0
    \qquad\text{a.e. on }\{\psi=c\}.
  \]
  The Sobolev chain rule therefore gives
  \begin{equation}\label{eq:vz-characteristic}
    v_\varepsilon
    =
    \mathds{1}_{\{|\psi|<\varepsilon\}}u,
    \qquad
    z_\varepsilon
    =
    \mathds{1}_{\{|\psi|>\varepsilon\}}u
    \qquad\text{a.e. in }\R^2.
  \end{equation}
  The level-set property removes the ambiguity on
  $\{\psi=\pm\varepsilon\}$.

  By \cref{lem:clamp},
  \begin{equation}\label{eq:v-uniform}
    \sup_{\varepsilon>0}
    \|v_\varepsilon\|_{\mathrm H^s}
    \leq
    C_s\sup_{\varepsilon>0}
    \|\mathcal C_\varepsilon(\psi)\|_{\mathrm H^{1+s}}
    \leq
    C_s\|\psi\|_{\mathrm H^{1+s}}
    <\infty.
  \end{equation}
  Consequently, $z_\varepsilon\in\mathrm H^s$.

  For each fixed $\varepsilon>0$, the decay of $\psi$ provides a radius
  $R>0$ such that
  \[
    |\psi(x)|<\varepsilon
    \qquad\text{whenever }|x|>R.
  \]
  Thus $\psi-\mathcal C_\varepsilon(\psi)$ vanishes outside $B_R$, and therefore
  $z_\varepsilon$ is compactly supported.

  \uline{Step 2:} \emph{A compactly supported divergence-free energy test.}
  Fix $\varepsilon>0$, let $\rho_\delta$ be a standard compactly supported
  mollifier, and set
  \[
    z_{\varepsilon,\delta}
    :=
    \rho_\delta*z_\varepsilon.
  \]
  Then
  \[
    z_{\varepsilon,\delta}\in\C_c^\infty(\R^2;\R^2),
    \qquad
    \operatorname{div} z_{\varepsilon,\delta}=0,
    \qquad
    z_{\varepsilon,\delta}\to  z_\varepsilon
    \quad\text{in }\mathrm H^s
  \]
  as $\delta\downarrow0$.  For $0<\delta<1$, the supports of
  $z_{\varepsilon,\delta}$ lie in one fixed compact set depending only
  on $\varepsilon$.

  Testing \cref{eq:NS} against $z_{\varepsilon,\delta}$ eliminates the
  pressure term.  The fractional term converges by $\mathrm H^s$ duality.
  The convection term also converges because $(u\cdot\nabla)u$ is smooth
  on the fixed compact set and
  $z_{\varepsilon,\delta}\to z_\varepsilon$ in $\mathrm L^2$.  Hence
  \begin{equation}\label{eq:test-z}
    \left\langle
    (-\Delta)^{s/2}u,
    (-\Delta)^{s/2}z_\varepsilon
    \right\rangle
    +
    \int_{\R^2}
    (u\cdot\nabla u)\cdot z_\varepsilon\,\, \mathrm{d} x
    =0.
  \end{equation}

  \uline{Step 3:} \emph{Exact cancellation of the convection term.}
  By \cref{eq:vz-characteristic},
  \[
    \int_{\R^2}
    (u\cdot\nabla u)\cdot z_\varepsilon\,\, \mathrm{d} x
    =
    \int_{\R^2}
    z_\varepsilon\cdot
    \nabla\left(\frac{|u|^2}{2}\right)\, \mathrm{d} x.
  \]
  Choose $\eta\in\C_c^\infty(\R^2)$ equal to one on a neighborhood of
  $\supp z_\varepsilon$.  Since $z_\varepsilon$ is compactly supported
  and divergence free in distributions,
  \[
    \begin{aligned}
      \int_{\R^2}
      z_\varepsilon\cdot
      \nabla\left(\frac{|u|^2}{2}\right)\, \mathrm{d} x
       & =
      \int_{\R^2}
      z_\varepsilon\cdot
      \nabla\left(\eta\frac{|u|^2}{2}\right)\, \mathrm{d} x \\
       & =
      -\left\langle
      \operatorname{div} z_\varepsilon,
      \eta\frac{|u|^2}{2}
      \right\rangle
      =0.
    \end{aligned}
  \]
From
  \cref{eq:test-z}, we obtain
  \begin{equation}\label{eq:orthogonality}
    \left\langle
    (-\Delta)^{s/2}u,
    (-\Delta)^{s/2}z_\varepsilon
    \right\rangle
    =0
    \qquad\text{for every }\varepsilon>0.
  \end{equation}

  \uline{Step 4:} \emph{Removal of the truncation.}
  By \cref{eq:vz-characteristic},
  $v_\varepsilon(x)\to0$ for a.e. $x\in\R^2$ as
  $\varepsilon\downarrow0$.  Indeed, this is immediate when
  $\psi(x)\neq0$, while on $\{\psi=0\}$ we have
  $u=\nabla^\perp\psi=0$ a.e. by the level-set property.
  Moreover,
  \(|v_\varepsilon|\leq|u|\) a.e. in $\mathbb R^2$.
  Since $u\in\mathrm L^2(\R^2;\R^2)$, Lebesgue's dominated convergence theorem gives
  \begin{equation}\label{eq:v-strong-L2}
    v_\varepsilon\to 0
    \qquad\text{strongly in }\mathrm L^2(\R^2;\R^2).
  \end{equation}
  Combining \cref{eq:v-uniform,eq:v-strong-L2}, we obtain
  \begin{equation}\label{eq:v-weak}
    v_\varepsilon\rightharpoonup0
    \qquad\text{weakly in }\mathrm H^s(\R^2;\R^2).
  \end{equation}
  Indeed, the uniform \(\mathrm H^s\)-bound gives weak subsequential
  compactness, and every weakly convergent subsequence has
  \(\mathrm L^2\)-weak limit zero by \cref{eq:v-strong-L2}.
 Thus zero is the only possible weak cluster point. If the full family did not converge weakly to zero, a sequence
  \(\varepsilon_k\downarrow0\) outside some weak neighborhood of zero would
  contradict this subsequence conclusion.

  Since $z_\varepsilon=u-v_\varepsilon$, it follows that
  $z_\varepsilon\rightharpoonup u$ weakly in $\mathrm H^s$.  Passing to
  the limit in \cref{eq:orthogonality} yields
  \(
  \|(-\Delta)^{s/2}u\|_{\mathrm L^2}^2 = 0.
  \)
  Since $u\in\mathrm L^2(\R^2;\R^2)$, Plancherel's theorem yields
  $u\equiv 0$ in $\mathbb R^2$.
\end{proof}

\section{\texorpdfstring{The annular cut-off argument for \(\frac13\leq s\leq\frac23\)}{The annular cut-off argument for 1/3 <= s <= 2/3}}
\label{sec:proof-sub}

The next proposition uses a direct annular cut-off argument.
For any \(0<s<1\), an additional \(L^r\)-bound with \(3\le r\le6\) forces triviality. In the range
\(\frac13\le s\le\frac23\), this bound follows directly from homogeneous
Sobolev embedding with \(r=q_s=2/(1-s)\). Thus the proposition proves the
\cref{thm:main-bounded} in that range and gives a second proof if 
\(\frac13\le s<\frac12\) already covered by \cref{sec:stream-range}.

\begin{proposition}\label{prop:low-range}
  Let \(0<s<1\), and let \((u,p)\) be a smooth solution of \cref{eq:NS} with
  \[
    u\in \Hdot^s(\R^2;\R^2).
  \]
  If \(u\in\mathrm L^r(\R^2;\R^2)\) for some \(3\le r\le6\), then
  \(u\equiv0\) and \(p\) is constant. In particular, the additional
  assumption is automatic when \(\frac13\le s\le\frac23\).
\end{proposition}

\begin{proof}
 Throughout the proof, \(p\) denotes the normalized pressure from
  \cref{lem:pressure-normalization}.
  
  \uline{Step 1.} \emph{Cutoff in the energy.} Let us fix $\chi\in \C_c^\infty(\R^2)$ such that
  \[
    0\le \chi\le1,
    \qquad \chi\equiv1\text{ in }B_1,
    \qquad \supp\chi\subset B_2,
    \qquad |\nabla\chi|\le C.
  \]
  We define $\chi_R(x)\coloneqq \chi(x/R)$, $A_R\coloneqq B_{2R}\setminus B_R$, and $\eta_R\coloneqq 1-\chi_R$.
  Testing the normalized equation \cref{eq:normalized-equation}
  against the compactly supported vector field $u\chi_R$ gives
  \begin{equation}\label{eq:test-identity}
    \langle (-\Delta)^s u, u\chi_R\rangle
    = -\int_{\R^2}(u\cdot\nabla u)\cdot(u\chi_R)\, \mathrm{d} x
    -\int_{\R^2}\nabla p\cdot(u\chi_R)\, \mathrm{d} x.
  \end{equation}
  Since $\operatorname{div} u=0$,
  \[
    \int_{\R^2}(u\cdot\nabla u)\cdot(u\chi_R)\, \mathrm{d} x
    = \frac12\int_{\R^2}u\cdot\nabla(|u|^2)\,\chi_R\, \mathrm{d} x
    = -\frac12\int_{\R^2}|u|^2\,u\cdot\nabla\chi_R\, \mathrm{d} x.
  \]
  Similarly,
  \[
    \int_{\R^2}\nabla p\cdot(u\chi_R)\, \mathrm{d} x
    = -\int_{\R^2}p\,u\cdot\nabla\chi_R\, \mathrm{d} x.
  \]
  Therefore
  \begin{equation}\label{eq:main-identity-R}
    \langle (-\Delta)^s u, u\chi_R\rangle
    = \frac12\int_{A_R}|u|^2\,u\cdot\nabla\chi_R\, \mathrm{d} x
    + \int_{A_R}p\,u\cdot\nabla\chi_R\, \mathrm{d} x.
  \end{equation}

  We first identify the limit of the left-hand side. Using the Fourier definition of $(-\Delta)^{s/2}$,
  \[
    \langle (-\Delta)^s u, u\chi_R\rangle
    = \int_{\R^2} (-\Delta)^{s/2}u \cdot (-\Delta)^{s/2}(u\chi_R)\, \mathrm{d} x.
  \]
  Since $u\chi_R = u-\eta_Ru$, \cref{lem:trunc} yields
  \[
    [\eta_Ru]_{\Hdot^s}\to0,
  \]
  and hence
  \begin{equation}\label{eq:left-limit}
    \langle (-\Delta)^s u, u\chi_R\rangle
    \to  [u]_{\Hdot^s(\R^2)}^2
    \qquad \text{as }R\to\infty.
  \end{equation}

  We now prove that the two terms on the right-hand side of \cref{eq:main-identity-R} vanish as $R\to\infty$.

  % We set $q\coloneqq q_s=2/(1-s)$, so that $u\in \mathrm{L}^q(\R^2)$ by \cref{eq:SobolevEmbedding}.

  %ALTERNATIVE; 

    \uline{Step 2.} \emph{The Bernoulli flux.}
  Set
  \[
    B\coloneqq p+\frac12|u|^2.
  \]
  Since \(u\in \mathrm L^r\), the Calder\'on--Zygmund estimate gives
  \(p\in \mathrm L^{r/2}\), and hence \(B\in \mathrm L^{r/2}\). The right-hand side of
  \cref{eq:main-identity-R} is the single flux
  \[
    \int_{A_R}B\,u\cdot\nabla\chi_R\,\mathrm dx.
  \]
  H\"older's inequality, \(|A_R|\lesssim R^2\), and
  \(|\nabla\chi_R|\lesssim R^{-1}\) give
  \begin{equation}\label{eq:annular-B-flux}
    \left|\int_{A_R}B\,u\cdot\nabla\chi_R\,\mathrm dx\right|
    \lesssim
    R^{1-6/r}\,
    \|B\|_{\mathrm L^{r/2}(A_R)}
    \|u\|_{\mathrm L^r(A_R)}.
  \end{equation}
  Because \(3\le r\le6\), the power of \(R\) is non-positive. Both annular
  norms tend to zero by absolute continuity of the corresponding global
  integrals, so the flux vanishes.

  Combining \cref{eq:main-identity-R}, \cref{eq:left-limit}, and
  \cref{eq:annular-B-flux} yields \([u]_{\Hdot^s}^2=0\). Hence \(u\) is
  constant, and its finite \(\mathrm L^r\)-norm forces \(u=0\). Returning to
  \cref{eq:NS} gives \(\nabla p=0\), so \(p\) is constant.

\end{proof}

\section{\texorpdfstring{Lorentz bootstrap and vorticity maximum
principle for \(\frac23<s<1\)}{Lorentz bootstrap and vorticity maximum principle for 2/3 < s < 1}}
\label{sec:proof-super}

For \(s>\frac23\), the annular estimate no longer closes. We instead bootstrap the projected equation to obtain boundedness and decay, then apply the nonlocal maximum principle to the vorticity. 

\begin{proposition}\label{prop:high-range}
  Let $\frac23<s<1$ and let $(u,p)$ be a smooth solution of \cref{eq:NS} with
  \[
    u\in \Hdot^s(\R^2;\R^2).
  \]
  Then $u\equiv0$ and $p$ is constant.
\end{proposition}

\begin{proof}
  We divide the proof into four steps.

  \uline{Step 1.} \emph{Bootstrap to boundedness.} This step is partly inspired by \cite{MR2763339,MR4813662}.
  We set
  \[
    \alpha \coloneqq  2s-1 \in \Bigl(\frac13,1\Bigr),
    \qquad q_0 \coloneqq q_s = \frac{2}{1-s}.
  \]
  By \cref{eq:SobolevEmbedding}, we have
  \[
    u\in \mathrm{L}^{q_0,2}(\R^2)
    \subset \mathrm L^{q_0}(\R^2),
  \]
  where the inclusion follows from \cite[eq.~(1.8)]{MR223874}.

  Suppose that
  \[
    u\in \mathrm{L}^{q_n,2}(\R^2)
  \]
  for some $n\ge0$. Then \cref{eq:lorentz-holder} yields
  \[
    u\otimes u\in \mathrm{L}^{q_n/2,1}(\R^2).
  \]
  Since $q_n>2$, we also have
  \[
    \mathrm L^{q_n,2}(\R^2)\subset\mathrm L^{q_n}(\R^2).
  \]
  If $q_n/2<2/\alpha$, \cref{lem:mapping-T}(a) gives
  \[
    T(u\otimes u)\in\mathrm L^{q_{n+1},1}(\R^2),
    \qquad
    \frac1{q_{n+1}}=\frac2{q_n}-\frac\alpha2.
  \]
Since \(a_0<\alpha/2\) and
\(a_{n+1}=2a_n-\alpha/2<a_n\) at every such step, induction gives
\(q_{n+1}>q_n\ge q_0>2\).
  Since $u\in\mathrm L^{q_n}(\R^2)$ and
  $T(u\otimes u)\in\mathrm L^{q_{n+1}}(\R^2)$, the
  finite-integrability case of \cref{lem:polynomial-remainder} shows that the
  polynomial in \cref{eq:u-equals-T} vanishes. Consequently,
  \[
    u\in\mathrm L^{q_{n+1},1}(\R^2)
    \subset\mathrm L^{q_{n+1},2}(\R^2),
  \]
  where the inclusion follows from \cite[(1.8)]{MR223874}.

  If $q_n/2=2/\alpha$, then \cref{lem:mapping-T}(b) yields
  \[
    T(u\otimes u)\in \mathrm{L}^\infty(\R^2).
  \]
  Since $u+T(u\otimes u)$ is a polynomial and
  $u\in\mathrm L^{q_n}(\R^2)$, \cref{lem:polynomial-remainder} implies that
  this polynomial is constant.
  Consequently,
  \begin{equation}\label{eq:uinfty}
    u\in \mathrm{L}^\infty(\R^2).
  \end{equation}

  If $q_n/2>2/\alpha$, then
  $u\otimes u\in \mathrm{L}^{q_n/2}(\R^2)$, so
  \cref{lem:mapping-T}(c) yields
  \[
    [T(u\otimes u)]_{\Cdot^{0,\alpha-4/q_n}}
    \lesssim
    \|u\otimes u\|_{\mathrm L^{q_n/2}}.
  \]
  Notice that this gives only a homogeneous H\"older bound, not an
  $\mathrm L^\infty$ bound. However, since
  \[
    u+T(u\otimes u)\in\mathcal P
  \]
  and $u\in\mathrm L^{q_n}(\R^2)$, \cref{lem:polynomial-remainder} implies that
  the polynomial remainder is constant. Therefore
  \[
    u\in \Cdot^{0,\alpha-4/q_n}(\R^2).
  \]
  In particular, $u$ is uniformly continuous. Since also
  $u\in\mathrm L^{q_n}(\R^2)$, \cref{lem:decay-uc}(a) gives
  $u(x)\to0$ as $|x|\to\infty$. Thus $u$ is bounded, and
  \cref{eq:uinfty} follows.

  It remains to see that one of the two latter alternatives is reached after
  finitely many subcritical steps. Let $a_n\coloneqq1/q_n$. As long as
  $q_n/2<2/\alpha$, the recurrence reads
  \[
    a_{n+1} = 2a_n - \frac\alpha2.
  \]
  Now
  \[
    a_0 = \frac{1-s}{2}
    < \frac{2s-1}{2} = \frac\alpha2
    \qquad \text{since }s>\frac23,
  \]
  so the sequence $a_n$ is strictly decreasing during the subcritical
  iteration. If the subcritical condition held at every step, then the
  recurrence would give
  \[
    a_n = \frac\alpha2 + 2^n\Bigl(a_0-\frac\alpha2\Bigr) \to -\infty,
  \]
  contradicting $a_n>\alpha/4$. Hence there is a first index $n_0$ such that
  \[
    a_{n_0}\le\frac\alpha4.
  \]
  If $n_0\ge1$, then
  $a_{n_0}=2a_{n_0-1}-\alpha/2>0$, so $q_{n_0}$ is still a finite positive
  exponent and $q_{n_0}/2\ge2/\alpha$.
  The endpoint case is handled by \cref{lem:mapping-T}(b), and the strict
  above-endpoint case is handled by \cref{lem:mapping-T}(c), together with
  \cref{lem:polynomial-remainder} and \cref{lem:decay-uc}(a).
  Therefore \cref{eq:uinfty} holds.

  \uline{Step 2.} \emph{Bootstrap to $\mathrm C^{1,\gamma}$.}
  We first obtain a positive H\"older exponent. Choose
  \[
    \delta\in\bigl(\max\{0,1-2\alpha\},\alpha\bigr)
    \qquad\text{with}\qquad
    \delta+\alpha\ne1.
  \]
  The interval is nonempty since $\alpha>1/3$, or equivalently $s>2/3$.
  Since $u\in \mathrm L^\infty(\R^2)\cap \mathrm L^{q_0}(\R^2)$, we have
  \[
    u\otimes u\in \mathrm L^m(\R^2)
  \]
  for all sufficiently large finite $m$. Taking $m$ so large that
  \[
    \alpha-\frac{2}m>\delta,
  \]
  \cref{lem:mapping-T}(c) and \cref{eq:u-equals-T} give
  \[
    T(u\otimes u)\in \Cdot^{0,\theta}(\R^2)
    \qquad\text{for some }\theta>\delta.
  \]
  Since $u\in \mathrm L^{q_0}(\R^2)$, \cref{lem:polynomial-remainder} again shows
  that the polynomial remainder in \cref{eq:u-equals-T} is constant. Hence
  \[
    u\in \Cdot^{0,\theta}(\R^2).
  \]
  Together with $u\in\mathrm L^\infty(\R^2)$, this implies
  \[
    u\in \mathrm C_b^{0,\delta}(\R^2).
  \]
  Therefore
  \[
    u\otimes u\in \mathrm C_b^{0,\delta}(\R^2).
  \]

  If $\delta+\alpha>1$, then \cref{lem:mapping-T}(d), applied to
  \cref{eq:u-equals-T}, gives
  \[
    u\in \mathrm C_b^{1,\delta+\alpha-1}(\R^2).
  \]

  If $\delta+\alpha<1$, then one application of \cref{lem:mapping-T}(d) gives
  \[
    u\in \mathrm C_b^{0,\delta+\alpha}(\R^2).
  \]
  Hence
  \[
    u\otimes u\in \mathrm C_b^{0,\delta+\alpha}(\R^2).
  \]
  Since our choice of $\delta$ gives
  \[
    \delta+2\alpha>1,
  \]
  a second application of \cref{lem:mapping-T}(d) gives
  \[
    u\in \mathrm C_b^{1,\delta+2\alpha-1}(\R^2).
  \]

  In both cases, there exists $\gamma>0$ such that
  \begin{equation}\label{eq:C1gamma}
    u\in \mathrm C_b^{1,\gamma}(\R^2).
  \end{equation}
  In particular, $u$ and $\nabla u$ are uniformly continuous on $\R^2$.

  \uline{Step 3.} \emph{Decay at infinity.}
  Since $u\in \mathrm{L}^{q_0}(\R^2)\cap \mathrm{UC}(\R^2;\R^2)$, \cref{lem:decay-uc}(a) gives
  \begin{equation}\label{eq:u-decay}
    u(x)\to0
    \qquad \text{as }|x|\to\infty.
  \end{equation}
  Applying \cref{lem:decay-uc}(b) to $f=u$, and using \cref{eq:C1gamma}, we also obtain
  \begin{equation}\label{eq:grad-decay}
    \nabla u(x)\to0
    \qquad \text{as }|x|\to\infty.
  \end{equation}
  Hence the vorticity $\omega=\operatorname{curl} u$ satisfies
  \begin{equation}\label{eq:omega-decay}
    \omega(x)\to0
    \qquad \text{as }|x|\to\infty.
  \end{equation}

  \uline{Step 4.} \emph{Maximum principle for the vorticity.}
  Taking the curl of \cref{eq:NS}, we obtain the scalar equation
  \[
    (-\Delta)^s\omega + u\cdot\nabla\omega =0
    \qquad \text{in }\R^2.
  \]
  By \cref{eq:C1gamma}, the function $\omega$ is continuous and bounded, and by \cref{eq:omega-decay}, $\omega(x)\to0$ as $|x|\to\infty$. In particular, any strictly positive supremum of $\omega$ is attained.

  Let us assume that $M\coloneqq \max_{\R^2}\omega>0$, and let $x_0$ be a point such that $\omega(x_0)=M$. Then, by smoothness,
  $\nabla\omega(x_0)=0$, and therefore
  \[
    0 = (-\Delta)^s\omega(x_0)+u(x_0)\cdot\nabla\omega(x_0)
    = (-\Delta)^s\omega(x_0).
  \]
  On the other hand, because $x_0$ is a global maximum point of the smooth,
  bounded function $\omega$, the integrand below is non-negative and integrable.
  Indeed, smoothness and $\nabla\omega(x_0)=0$ give local integrability because
  $s<1$, while boundedness gives integrability at infinity. Thus the principal
  value is an ordinary integral:
  \[
    (-\Delta)^s\omega(x_0)
    = c_{2,s}\int_{\R^2}\frac{\omega(x_0)-\omega(y)}{|x_0-y|^{2+2s}}\, \mathrm{d} y
    \ge 0.
  \]
  Since the integral is both non-negative and equal to zero, we must have $\omega(y)=M$ for a.e.~$y\in\R^2$. By continuity, $\omega\equiv M$, which contradicts \cref{eq:omega-decay}. Therefore $\max\omega\le0$.
  Applying the same argument to $-\omega$ yields $\min\omega\ge0$. Hence
  \[
    \omega\equiv0.
  \]
  Thus $u$ is both divergence-free and irrotational. In dimension two this implies that each component of $u$ is harmonic. Since $u\in \mathrm{L}^{q_0}(\R^2)$ with $q_0<\infty$, the classical Liouville theorem for harmonic functions yields $u\equiv0$.
\end{proof}

\begin{remark}
  We have shown that if $\frac23<s<1$ and $(u,p)$ is a smooth solution of
  \cref{eq:NS} with $u\in\Hdot^s(\R^2;\R^2)$, then
  $u\in\mathrm L^\infty(\R^2;\R^2)$.
  In this case one also obtains
  \[
    \nabla u\in \mathrm L^{2}(\R^2).
  \]
  Indeed, from the projected equation
  \[
    (-\Delta)^s u=-\PP\operatorname{div}(u\otimes u),
  \]
  taking Fourier transforms and multiplying by $|\xi|^{s-1}$ for
  $\xi\ne0$, we obtain
  \[
    (-\Delta)^{\frac{3s}{2}-\frac{1}{2}}u
    =
    -(-\Delta)^{\frac{s}{2}}
    (-\Delta)^{-\frac{1}{2}}\PP\operatorname{div}(u\otimes u)
    \qquad\text{on }\R^2\setminus\{0\}.
  \]
  The operator
  $(-\Delta)^{-1/2}\PP\operatorname{div}$ is a Calder\'on--Zygmund operator of order
  zero. The fractional Leibniz rule (see \cite{MR4333988}) therefore gives
  \[
    \begin{aligned}
      \bigl\|(-\Delta)^{\frac{3s}{2}-\frac{1}{2}}u\bigr\|_{\mathrm L^2}
       & \lesssim
      \bigl\|(-\Delta)^{\frac{s}{2}}(u\otimes u)\bigr\|_{\mathrm L^2}
      \\
       & \lesssim
      \|u\|_{\mathrm L^\infty}
      \bigl\|(-\Delta)^{\frac{s}{2}}u\bigr\|_{\mathrm L^2}
      <\infty.
    \end{aligned}
  \]
  Thus the identity away from the origin defines an $\mathrm L^2$ representative
  of $(-\Delta)^{(3s-1)/2}u$; the possible polynomial ambiguity is excluded by
  $u\in\mathrm L^{q_0}(\R^2)$. Therefore
  \[
    u\in \Hdot^{3s-1}(\R^2).
  \]
  Since $\frac23<s<1$, we have $s<1<3s-1$. Interpolating between
  $\Hdot^s$ and $\Hdot^{3s-1}$ (see, for example,
  \cite[Proposition~1.32]{CDBOOK}) gives
  \[
    \|\nabla u\|_{\mathrm L^2}
    \lesssim
    \bigl\|(-\Delta)^{\frac{3s}{2}-\frac{1}{2}}u\bigr\|_{\mathrm L^2}^{\frac{1-s}{2s-1}}
    \bigl\|(-\Delta)^{\frac{s}{2}}u\bigr\|_{\mathrm L^2}^{\frac{3s-2}{2s-1}}
    <\infty.
  \]
  This implies $\omega\in\mathrm L^2(\R^2)$. Let
  $\omega_\varepsilon=\rho_\varepsilon*\omega$, where
  $\rho_\varepsilon$ is a standard mollifier, and set
  \[
    r_\varepsilon
    \coloneqq
    u\cdot\nabla\omega_\varepsilon
    -\rho_\varepsilon*(u\cdot\nabla\omega).
  \]
  Since \cref{eq:C1gamma} implies
  $u\in\mathrm W^{1,\infty}(\R^2)$, the Friedrichs commutator lemma gives
  \[
    r_\varepsilon\to0
    \qquad\text{in }\mathrm L^2(\R^2).
  \]
  Mollifying the vorticity equation yields
  \[
    (-\Delta)^s\omega_\varepsilon
    +u\cdot\nabla\omega_\varepsilon
    =r_\varepsilon.
  \]
  Testing this equation with $\omega_\varepsilon$ and using
  $\operatorname{div}u=0$, justified by a spatial cutoff, gives
  \[
    \|(-\Delta)^{s/2}\omega_\varepsilon\|_{\mathrm L^2}^2
    =
    \int_{\R^2}r_\varepsilon\omega_\varepsilon\,\, \mathrm{d} x
    \to0.
  \]
  Fatou's lemma in Fourier space now gives
  $(-\Delta)^{s/2}\omega=0$. Since
  $\omega\in\mathrm L^2(\R^2)$, it follows that $\omega=0$. We conclude as
  in the proof of \cref{prop:high-range}.
\end{remark}

\section{Proof of \texorpdfstring{\cref{thm:main-bounded}}{the Liouville theorem in case 0 < s < 1}}
\label{sec:proofs}

Putting the arguments of \crefrange{sec:stream-range}{sec:proof-super} together, we can now prove \cref{thm:main-bounded}.

\begin{proof}[Proof of \cref{thm:main-bounded}]
  By \cref{prop:stream-range,prop:low-range,prop:high-range}, respectively,
  \(u=0\) for \(0<s<\frac12\), for
  \(\frac13\le s\le\frac23\) (where \(u\in L^{q_s}\) and
  \(3\le q_s\le6\)), and for \(\frac23<s<1\). These overlapping ranges
  cover \((0,1)\). Substitution into \eqref{eq:NS} gives \(\nabla p=0\), so
  \(p\) is constant.
\end{proof}

\section{Proof of \texorpdfstring{\cref{thm:damped-euler}}{the Liouville theorem in case s = 0}}
\label{sec:euler}

The proof of \cref{thm:damped-euler} relies on a Bernoulli identity with a
non-negative dissipation term. Testing this identity against expanding
cutoffs, the assumption \(u\in\mathrm L^2(\R^2)\) makes the boundary term
vanish and forces \(u\equiv0\).

\begin{proof}
  \uline{Step 1.} \emph{The dissipative Bernoulli identity.}
  Introduce the Bernoulli function
  \[
    B\coloneqq p+\frac12|u|^2.
  \]
  Taking the scalar product of the momentum equation in
  \cref{eq:euler} with \(u\) gives
  \begin{equation}\label{eq:damped-bernoulli}
    u\cdot\nabla B=-|u|^2.
  \end{equation}
  Since \(\operatorname{div} u=0\), the chain rule and
  \cref{eq:damped-bernoulli} yield
  \begin{equation}\label{eq:damped-bernoulli-div}
    \operatorname{div}\bigl(u\,\arctan(B)\bigr)
    =
    -\frac{|u|^2}{1+B^2}.
  \end{equation}

  \uline{Step 2.} \emph{The flux vanishes at infinity.}
  \(\chi=1\) on \(B_1\) and \(\supp\chi\subset B_2\), and set
  \(\chi_R(x)=\chi(x/R)\) and \(A_R=B_{2R}\setminus B_R\).
  Testing \cref{eq:damped-bernoulli-div} against \(\chi_R\) gives
  \begin{equation}\label{eq:damped-euler-test}
    0\le
    \int_{\R^2}\chi_R\frac{|u|^2}{1+B^2}\,\mathrm dx
    =
    \int_{\R^2}\arctan(B)\,u\cdot\nabla\chi_R\,\mathrm dx
    \le
    \left|
    \int_{\R^2}\arctan(B)\,u\cdot\nabla\chi_R\,\mathrm dx
    \right|
    \le
    \frac{C}{R}\int_{A_R}|u|\,\mathrm dx .
  \end{equation}
  By \cref{eq:damped-annular-condition}, there is a sequence
  \(R_k\to\infty\) along which the right-hand side tends to zero. Since
  \(\chi_{R_k}\to1\) pointwise and the integrands on the left are
  nonnegative, Fatou's lemma gives
  \[
    0\le
    \int_{\R^2}\frac{|u|^2}{1+B^2}\,\mathrm dx
    \le
    \liminf_{k\to\infty}
    \int_{\R^2}\chi_{R_k}\frac{|u|^2}{1+B^2}\,\mathrm dx
    =0.
  \]
  Hence \(u=0\), and \cref{eq:euler} gives \(\nabla p=0\), so \(p\) is
  constant.

  Finally, if \(u\in L^r(\R^2)\) for some \(1\le r\le2\), then
  \[
    \frac1R\int_{A_R}|u|\,\mathrm dx
    \lesssim R^{1-2/r}\|u\|_{L^r(A_R)}
    \to 0.
  \]
\end{proof}

\vspace{0.5cm}
\section*{Acknowledgments}

N.~De Nitti is a member of the Gruppo Nazionale per l'Analisi Matematica, la Probabilità e le loro Applicazioni (GNAMPA) of the Istituto Nazionale di Alta Matematica (INdAM) and has received support from the INdAM--GNAMPA Project 2026 \textit{Modelli Non-locali in Fluidodinamica, Traffico ed Elasticità} (CUP:~E53C25002010001).

He is grateful to X.~Fernández-Real, F.~Hounkpe, and S.~Schulz for helpful conversations on topics related to this work.

L.~Niebel is funded
by the Deutsche Forschungsgemeinschaft (DFG, German Research Foundation) under Germany's
Excellence Strategy EXC 2044/2--390685587, Mathematics M\"unster: Dynamics--Geometry--Structure.

J.~Yang is funded by the National Natural Science Foundation of China (NSFC) under grant 12471225, and  Natural Science Basic Research Program of Shaanxi (Program No.~2026JC-YXQN-021).

\vspace{0.5cm}

\printbibliography

\vfill

\end{document}